\documentclass[12pt]{amsart}
\usepackage[left=1.4in, right=1.4in, top=1.4in, bottom=1.4in]{geometry}
\usepackage{amssymb}
\usepackage{mathtools,xspace}

\usepackage{titletoc}
\usepackage[shortlabels]{enumitem}

\usepackage{mathrsfs}
\usepackage{bbm}
\usepackage{bm}

\usepackage[usenames,dvipsnames]{xcolor}
\definecolor{darkblue}{rgb}{0.0, 0.0, 0.55}
\definecolor{bordeaux}{rgb}{0.34, 0.01, 0.1}
\definecolor{brightbordeaux}{rgb}{0.68, 0.14, 0.14}

\usepackage[pagebackref,colorlinks,linkcolor=brightbordeaux,citecolor=darkblue,urlcolor=black,hypertexnames=true]{hyperref}

\usepackage{nag}

\usepackage{ulem}

\renewcommand{\subset}{\subseteq}

\newtheorem{theorem}            {Theorem}[section]
\newtheorem{corollary}          [theorem]{Corollary}

\theoremstyle{definition}
\newtheorem{definition}         [theorem]{Definition}

\theoremstyle{theorem}

\theoremstyle{definition}
\newtheorem{example}            [theorem]{Example}

\newcommand{\norm}[1]{{ \Vert #1 \Vert }}

\def\B{ {\mathcal B} }

\def\cH{ {\mathcal H} }

\newcommand{\vf}{{\varphi}}

\def\C{{\mathbb C}}

\newcommand{\alg}{\operatorname{Alg}}

\mathchardef\mhyphen="2D

\newcommand{\bsbm}{\left[ \begin{smallmatrix}}
\newcommand{\esbm}{\end{smallmatrix} \right]}

\newcommand{\bbm}{ \begin{bmatrix}}
\newcommand{\ebm}{\end{bmatrix} }

\newcommand{\bpm}{\begin{pmatrix}}
\newcommand{\epm}{\end{pmatrix}}

\newcommand{\bspm}{\left(\begin{smallmatrix}}
\newcommand{\espm}{\end{smallmatrix}\right)}

\newcommand{\bsm}{\begin{smallmatrix}}
\newcommand{\esm}{\end{smallmatrix}}

\newcommand{\bal}{\begin{align*}}
\newcommand{\eal}{\end{align*}}

\title[Algebraicity of operator NC functions]{Noncommutative functions are weakly algebraic on operatorial polynomial polyhedra}

\author[Mark E. Mancuso]{Mark E. Mancuso}
\address{Mark E. Mancuso, Department of Mathematics, Lafayette College, Easton}
\email{mancusom@lafayette.edu}

\subjclass[2010]{46L89, 47A56 (Primary); 46A22, 47A60 (Secondary)}

\keywords{noncommutative function theory, operatorial noncommutative function, operatorial polynomial polyhedron, weak operator topology, weakly algebraic}

\numberwithin{equation}{section}

\begin{document}

\begin{abstract}
An operatorial polynomial polyhedron is a set of the form $$B_{\delta}(\B(\cH))=\{X\in \B(\cH)^d : \Vert\delta(X)\Vert<1\}$$ where $\B(\cH)$ denotes the space of bounded operators on a separable Hilbert space $\cH$, and $\delta$ is a matrix of polynomials in $d$ noncommuting variables. These sets appear throughout the literature on noncommutative function theory.  While much of what has been written involves matricial polynomial polyhedra, there do exist $\delta$ such that the associated $B_{\delta}(\B(\cH))$ is non-empty but contains no matrix points. Algebraicity of operatorial noncommutative functions has  been established in the case that the domain $B_{\delta}(\B(\cH))$ is a balanced set (hence contains the matrix point 0). In this paper, we dispense of such assumptions on the domain and prove that an operatorial noncommutative function on any $B_{\delta}(\B(\cH))$ is weakly algebraic in the sense that its value at each operator tuple $Z$ lies in the weak operator topology closure of the unital algebra generated by the coordinates of $Z$.
\end{abstract}

\maketitle

\section{Introduction}

Free polynomials (polynomials in $d$ noncommuting variables with complex coefficients) may be evaluated at $d$-tuples of square complex matrices of arbitrary size. Such evaluation of a free (noncommutative) polynomial gives rise to a function that is {\it graded} in the sense that $d$-tuples of $n\times n$ matrices are mapped to $n\times n$ matrices. Moreover, straightforward algebraic calculations show that a noncommutative polynomial function $p$ must {\it preserve direct sums and
similarities:} $$p(X\oplus Y)=p(X)\oplus p(Y)\,\,\,\, \text{and} \,\,\,\, p(s^{-1}Xs)=s^{-1}p(X)s$$ for matrix tuples $X$ and $Y$ and an invertible matrix $s$.  There are many other examples of functions besides polynomials that may be evaluated at tuples of matrices of arbitrary size and that satisfy these algebraic properties; such functions are called {\it noncommutative functions}. Noncommutative functions form the foundation for the burgeoning field of noncommutative function theory. 

We provide an overview of the classical {\it matricial} case of noncommutative function theory; the {\it operatorial} case (which is relevant for this paper's main results) will be detailed subsequently. Matricial noncommutative functions are defined on domains that are subsets of the graded space of $d$-tuples of square matrices of arbitrary size. For our purposes, such matrices have entries in the field of complex numbers. That is, we begin with the graded space $$M(\mathbb{C})^d=(M_n(\mathbb{C})^d)_{n=1}^{\infty}$$ where $M_n(\mathbb{C})$ denotes the set of $n\times n$ complex matrices. Thus, an element $X$ of $M(\mathbb{C})^d$ is a $d$-tuple of complex matrices, which we will write using superscripts: $X=(X^1,\ldots,X^d)$. A subset $D\subset M(\mathbb{C})^d$ is called a {\it noncommutative domain} if it is open at each level (i.e. each $D\cap M_n(\mathbb{C})^d$ is open in the Euclidean topology) and is closed under direct sums in the sense that if $X\in D\cap M_n(\mathbb{C})^d$ and $Y\in D\cap M_m(\mathbb{C})^d$, then $X\oplus Y\in M_{n+m}(\mathbb{C})^d.$ For some purposes, authors may also require that each level is closed under conjugation by unitary matrices: if $X\in D\cap M_n(\mathbb{C})^d$ and $U\in \mathcal{U}_n$ is an $n\times n$ unitary matrix, then $U^*XU \in D\cap M_n(\mathbb{C})^d$. By $U^*XU$ we mean the $d$-tuple $U^*(X^1,\ldots,X^d)U=(U^*X^1U,\ldots, U^*X^dU)$. The direct sum $X\oplus Y$ is similarly defined componentwise as $(X^1\oplus Y^1,\ldots, X^d\oplus Y^d)$.

A large and important class of examples of matricial noncommutative domains can be described as follows. Let $\delta$ be an $I\times J$ matrix of  noncommutative polynomials in $d$ variables. We define the (matricial) polynomial polyhedron associated to $\delta$ to be the set $$B_{\delta}=\{X\in M(\mathbb{C})^d : \Vert \delta(X)\Vert <1\}.$$ Such domains have been studied, for example, in \cite{amop, global1, global2} in the context of noncommutative realization theory and in \cite{rowball} in the context of Nevanlinna-Pick interpolation.

A classical {\it matricial noncommutative function} is a function $$f:D\subset M(\mathbb{C})^d\rightarrow M(\mathbb{C})^1$$ defined on a noncommutative domain $D$ that satisfies:
\begin{itemize}
\item $f$ is a graded function: if $X\in D\cap M_n(\mathbb{C})^d$, then $f(X)\in M_n(\mathbb{C})$.
\item $f$ preserves direct sums: if $X\in D\cap M_n(\mathbb{C})^d$ and $Y\in D\cap M_m(\mathbb{C})^d$, then $f(X\oplus Y)=f(X)\oplus f(Y)$.
\item $f$ preserves similarities: if $X \in D\cap M_n(\mathbb{C})^d$ and $s\in M_n(\mathbb{C})$ is an invertible $n\times n$ matrix such that $s^{-1}Xs\in D$, then $f(s^{-1}Xs)=s^{-1}f(X)s$.
\end{itemize} 
Note that the requirement that $f$ is a graded function is needed for the next two requirements to make sense. We refer the reader to \cite{book} for a comprehensive treatment on the foundations of the matricial noncommutative function theory. Such functions have also been studied in the context of noncommutative inversion in \cite{abd, HKM, james}. One remarkable property (see \cite{opanalysis}) of matricial noncommutative functions is that local boundedness automatically implies (complex) differentiability.

Recently \cite{amop, opNC, JKMMP, inverse}, it has become beneficial (and interesting in its own right) to consider a completion of the aforementioned matricial setting. By this we mean that we want to study noncommutative functions that are defined on domains $\Sigma \subset \B(\cH)^d$, where $\cH$ is an infinite-dimensional separable Hilbert space. Such domains and functions are now called {\bf operator NC domains} and {\bf operator NC functions}, respectively. In other words, in this operatorial formulation of noncommutative function theory, we replace the graded space $M(\mathbb{C})^d$ of matrix tuples with the complete space $\B(\cH)^d$ of operator tuples. Operator NC functions preserve direct sums and similarities in a certain precise sense, but since the domain of such a function is not graded (indeed, it lies inside of the Banach space $\B(\cH)^d$), one must make identifications of the underlying Hilbert space $\cH$ with finite or countably infinite direct sums of itself by way of unitary equivalence. Though care is needed to fully define an operator NC domain, such a domain will satisfy in particular the following analogue of being closed under direct sums: if $X,Y \in \Sigma$ then there exists a unitary operator $U:\cH\rightarrow \cH^{(2)}$ such that $$U^{-1}\begin{bmatrix}
X & 0 \\
0 & Y
\end{bmatrix} U\in \Sigma,$$ where we regard the direct sum of $X$ and $Y$ as an operator in $\B(\cH^{(2)})$. An operator NC function $f:\Sigma \rightarrow \B(\cH)$ would then act upon such an element as follows: 
$$f\left(U^{-1}\begin{bmatrix}
X & 0 \\
0 & Y
\end{bmatrix} U\right)=U^{-1}\begin{bmatrix}
f(X) & 0 \\
0 & f(Y)
\end{bmatrix} U.$$ 

The precise definitions of general operator NC domains and operator NC functions, as proposed by the author in \cite{inverse}, are provided in Section 2. These definitions were subsequently used in \cite{opanalysis, opNC}. For now, we are content with providing a few elementary examples of these objects.

\begin{example}
(1) The polynomial function $$f(X,Y)=8iXYX+XY-YX$$ is an operator NC function on the entire operator NC domain $\B(\cH)^2.$

(2) The rational function $$g(X,Y,Z)=Y^{-1}X(1-XYZ)^{-1}$$ defines an operator NC function on the operator NC domain
 \begin{equation}
\{(X,Y,Z)\in \B(\cH)^3: \Vert 1-Y\Vert<1 \,\, \text{and} \,\, \Vert XYZ\Vert<1\}.
\end{equation}

(3) The function of one variable $$h(X)=X^{-1}$$ is an operator NC function on the operator NC domain of all invertible elements of $\B(\cH)$.
\end{example}

The main purpose of this note is not to consider noncommutative functions on general operatorial noncommutative domains, but instead is to focus on a particularly nice class of domains called operatorial polynomial polyhedra.  Let $\delta$ be an $I\times J$  matrix of noncommutative polynomials in $d$ variables. That is, $\delta\in  M_{I\times J}(\C\langle x_1,\ldots , x_d\rangle)$.  We define its associated 
{\bf operatorial polynomial polyhedron} $B_{\delta}(\B(\cH))$ by 
\[
	B_{\delta}(\B(\cH))=\{X\in \B(\cH)^d : \norm{\delta(X)}<1\}.
\]
The norm of $\delta(X)$ is understood to take place in the space $\B(\cH^{(J)},\cH^{(I)}).$ Note that there is no harm in assuming for simplicity that $\delta$ is square ($I=J$); simply add rows or columns of zeros. 

Many important concrete examples of noncommutative operator domains arise as a $B_{\delta}(\B(\cH))$ for some choice of $\delta$. 

\begin{example}
(1) The noncommutative operatorial polydisk $$\B_1(\cH)^d=\{X\in \B(\cH)^d : \max_{1\leq i\leq d} \Vert X^i\Vert <1 \}$$ is a $B_{\delta}(\B(\cH))$ where $\delta$ is the block diagonal matrix with diagonal entries the coordinates of $X$. 

(2) Another important and well-studied example (see for example \cite{pop1, pop2}) is the noncommutative operatorial row ball. That is, if we let $$\mathcal{R}(\B(\cH))=\{X\in \B(\cH)^d : \Vert X^1(X^1)^*+\cdots +X^d(X^d)^*\Vert^{1/2}<1\},$$ then $\mathcal{R}(\B(\cH))$ is readily seen to be $B_{\delta}(\B(\cH))$ where $\delta$ is taken to be the row $\delta(x)=[x^1\cdots x^d]$.

(3) The operator NC domain appearing in (1.1) is the intersection of two operatorial polynomial polyhedra and is thus one itself.
\end{example}

 We say a function $\vf :B_{\delta}(\B(\cH)) \rightarrow \B(\cH)$ is {\bf weakly algebraic}  if, for each operator tuple $Z\in B_{\delta}(\B(\cH))$, the value $\vf(Z)$ lies in the weak operator topology (WOT) closure of the unital subalgebra $\alg(Z)$ of $\B(\cH)$ generated by the coordinates of $Z$. That is, if $$\vf(Z) \in \overline{\alg(Z)}^{{}_{WOT}}$$ for each $Z\in B_{\delta}(\B(\cH))$. The main result of this paper, Theorem 1.3 below, (which is proved as Theorem 3.1) says that an operator NC function defined on any operatorial polynomial polyhedron is weakly algebraic. 

\begin{theorem}
Every operator NC function $\vf :B_{\delta}(\B(\cH)) \rightarrow \B(\cH)$ is weakly algebraic.
\end{theorem}

Other versions of statements along these lines have been established by Agler and McCarthy in \cite{amop, global2} and recently by Augat and McCarthy in \cite{opNC}. In \cite{global2}, the authors proved a {\it matricial} version of algebraicity of noncommutative functions of polynomial polyhedra. Variants of the operatorial case are considered in both \cite{amop, opNC}. Crucially, in both of these treatments of the operatorial case, the authors rely heavily on the assumption that the polynomial polyhedron $B_{\delta}(\B(\cH))$ is a {\it balanced set}. Recall that a subset $S$ of a vector space is said to be balanced if $\alpha S\subset S$ for every scalar $\alpha$ with $\vert \alpha \vert \leq 1$. That is, all previous results on algebraicity of operator NC functions on polynomial polyhedra known to the author require that the domain is balanced. Of course, any balanced set must in particular contain 0. The novelty in the present paper is that we are able to dispense of such an assumption. Indeed, weak algebraicity is established for any $B_{\delta}(\B(\cH))$, even for those that contain no matrix points. An example of a $B_{\delta}(\B(\cH))$ containing no matrix points is given in Section 2.

\section{Preliminaries}

This section contains the precise notions of operator NC domain and function as well as a discussion of known results on algebraicity of noncommutative functions. In what follows, we fix $\cH$, a separable infinite-dimensional Hilbert space over the complex numbers. Definitions 2.1 and 2.2 were first proposed in \cite{inverse} in the context of operatorial noncommutative inverse and implicit function theorems, and we refer the reader to that article for more comments on these definitions.

 \begin{definition} 
  A subset $\Sigma\subset \B(\cH)^d$ is called an {\bf operator NC domain} if there exists an exhausting sequence $(\Sigma_k)_{k=1}^{\infty}$ of subsets of $\Sigma$ with the following properties:

\begin{enumerate}
  \item  $\Sigma=\bigcup_{k=1}^{\infty} \Sigma_k$ and $\Sigma_k \subset$  int $\Sigma_{k+1}$ for all $k\geq 1$.

  \item Every term $\Sigma_k$ is bounded and closed under unitary conjugation.

  \item Every term $\Sigma_k$ is closed under countable direct sums in the sense that if $X_n$ is a sequence in $\Sigma_k$ of length $l\in \mathbb{N}\cup \{\infty\}$, then there exists a unitary operator $U:\cH\rightarrow \cH^{(l)}$ such that \begin{align} U^{-1}\begin{bmatrix}
    X_{1} & &  \\
    & X_{2} & \\
    & & \ddots  
  \end{bmatrix}U \in \Sigma_k.\end{align}
\end{enumerate}
  \end{definition}

This definition of operator NC domain is quite general, but in this paper, we are primarily concerned with the special case of polynomial polyhedra. Let $\delta$ be an $I\times J$  matrix of noncommutative polynomials in $d$ variables, and recall that its associated 
operatorial polynomial polyhedron is  
\[
	B_{\delta}(\B(\cH))=\{X\in \B(\cH)^d : \norm{\delta(X)}<1\}.
\]
 Crucially, every operatorial polynomial polyhedron is an example of an operator NC domain. Indeed, one may choose the exhausting sequence $(B_{\delta}(\B(\cH))^k)_k$ where \begin{equation}
B_{\delta}(\B(\cH))^k=\{X\in \B(\cH)^d : \norm{\delta(X)}\leq 1-1/k \,\, \text{and}\,\, \norm{X}\leq k\}.\end{equation} We use the max norm on $\B(\cH)^d$ given by $\Vert X\Vert := \max_{1\leq i \leq d} \Vert X^i \Vert.$ The reader is invited to check that the exhausting sequence in (2.2) satisfies the requirements in Definition 2.1.

  We now turn our attention to {\it functions} that are defined on operator NC domains $\Sigma \subset \B(\cH)^d$. 

  \begin{definition} 
  Suppose $\Sigma \subset \B(\cH)^d$ is an operator NC domain. A function $f:\Sigma \rightarrow \B(\cH)$ is said to be an {\bf operator NC function} if whenever $X,Y \in \Sigma$ and  $s:\cH\rightarrow \cH^{(2)}$ is bounded and invertible such that $$s^{-1}\begin{bmatrix}
    X & 0 \\
   0 & Y 
  \end{bmatrix}s \in \Sigma,$$ then $$f\left(s^{-1}\begin{bmatrix}
    X & 0  \\
   0 & Y   
  \end{bmatrix}s\right)=s^{-1}\begin{bmatrix}
    f(X) & 0  \\
   0 & f(Y)  
  \end{bmatrix}s.$$
  \end{definition}

It was shown in \cite{inverse} that operator NC functions (when defined on operator NC domains as in Definition 2.1) automatically preserve {\it countable} direct sums and intertwinings. Moreover, they are also always {\it locally bounded} on operator NC domains. Indeed, they must in fact be bounded on each term $\Sigma_k$ of an exhausting sequence $(\Sigma_k)_{k=1}^{\infty}$ of the operator NC domain. This is a particularly nice feature of working with domains that have exhaustions that satisfy (2.1). It alleviates the need to assume {\it a priori} that the noncommutative function with which one is working is locally bounded. Such an assumption is pervasive throughout the noncommutative function theory literature.

We conclude this section by collecting some known results on the algebraicity of noncommutative functions in various settings and contrasting them with the main result of this paper. In the {\it matricial} setting, the problem of pointwise algebraicity was settled by Agler and McCarthy:

\begin{theorem}[\cite{global2}, Theorem 2.4]
If $\vf \in H^{\infty}(B_{\delta})$, then $\vf(z) \in \alg(z)$ for each $z \in B_{\delta}$.
\end{theorem}

By $H^{\infty}(B_{\delta})$ the authors mean the space of bounded matricial noncommutative functions on $B_{\delta}$. In the case of operator NC functions, one consequence of Theorem 3.1 is that one need not  assume boundedness of the function to conclude weak algebraicity. 

Agler and McCarthy also considered algebraicity questions on polynomial polyhedra in the context of {\it operatorial} noncommutative functions in \cite{amop}. In that paper, the authors relied on the additional assumption that the operator NC function was sequentially continuous in the strong operator topology (SOT). Structure in the SOT in the context of operatorial noncommutative analysis has also been studied in \cite{JKMMP, inverse}. The main result in \cite{amop} on algebraicity was recently refined by Augat and McCarthy in \cite{opNC} who were able to omit the sequential SOT-continuity hypothesis. In our notation, their result is as follows:

\begin{theorem}[\cite{opNC}, Corollary 4.4]
Suppose $F: B_{\delta}(\B(\cH)) \rightarrow \B(\cH)$ is an operator NC function on a balanced domain $B_{\delta}(\B(\cH))$. Then $F(X)$ is in the norm-closed unital algebra $\alg(X)$.
\end{theorem}

As mentioned in Section 1, Theorem 2.4 requires the operatorial polynomial polyhedron to be a balanced set. In particular, it must contain the matrix point 0. 
In turn, with this extra hypothesis, they are able to conclude that the function values lie in the norm-closed algebra. Not all $B_{\delta}(\B(\cH))$ are balanced, however:

\begin{example}
Let $\delta \in M_{1\times 1}(\mathbb{C}\langle x,y\rangle)$ be given by $$\delta(x,y)=1-[x,y]=1-(xy-yx).$$ Then the associated operatorial $B_{\delta}(\B(\cH))$ is not balanced and non-empty. In fact, it does not contain a single matrix point. In other words, the associated {\it matricial} $B_{\delta}$ is empty.
\end{example} 

Thus, in particular, Theorem 2.4 does not apply to operator NC functions defined on the $B_{\delta}(\B(\cH))$ appearing in Example 2.5. On the other hand, Theorem 1.3 does apply in this case.

\section{Main Results}

In this section, we state and prove the main results of this paper. Recall that $\alg(Z)$ denotes the unital subalgebra 
of $\B(\cH)$ generated by the coordinates of $Z\in \B(\cH)^d$.  

\begin{theorem}
	If $\vf :B_{\delta}(\B(\cH)) \rightarrow \B(\cH)$ is an operator NC function, then $$\vf(Z)\in \overline{\alg(Z)}^{{}_{WOT}}$$ for every $Z\in B_{\delta}(\B(\cH))$.
\end{theorem}

We first prove a separation-type theorem for the WOT-closure of the algebra generated by the coordinates of $Z$ that may be seen as a stepping-stone to Theorem 3.1, but it is also interesting in its own right. We write $Z^{(n)}$ to denote the direct sum of $n$ copies of the operator tuple $Z.$ Theorem 3.2 below is essentially an application of the Hahn-Banach theorem in the context of the locally convex topological vector space $\B(\cH)$ endowed with the weak operator topology. The method of proof of Theorem 3.2 can be seen as a WOT version of some ideas used in \cite{global2}. Recall that we use the max norm for the norm of an operator tuple.

\begin{theorem}
	If $Z\in \B_1(\cH)^d$ and $W\notin \overline{\alg(Z)}^{{}_{WOT}}$  there exists a positive integer $n$ and invertible $s:\cH\rightarrow 
	\cH^{(n)}$ such that $\Vert s^{-1}Z^{(n)}s\Vert<1$ and $\Vert s^{-1}W^{(n)}s\Vert>1$.
\end{theorem}

\begin{proof}
	By the Hahn-Banach theorem, there is a WOT-continuous linear functional $\ell: \B(\cH)\rightarrow \C$ such that 
	$\ell(\overline{\alg(Z)}^{{}_{WOT}})=0$ and $\ell(W)\neq0.$ There is a positive integer $n$ and vectors $u_1,\ldots, u_n$ and $v_1,\ldots, 
	v_n$ in $\cH$ such that
	\[
		\ell(T)=\sum_{j=1}^n\langle Tu_j,v_j\rangle
	\]
	for $T\in \B(\cH)$. (See Proposition 5.1 in \cite{conway} for a reference on the characterization of WOT-continuous linear functionals on $\B(\cH)$.)
	
	Let $\alg(Z)^{(n)}$ denote the set of $n$-ampliations of elements of $\alg(Z)$. That is, let $\alg(Z)^{(n)}:= \{S^{(n)}:S\in \alg(Z)\}$. Now define a closed subspace $\mathcal M$ of $\cH^{(n)}$ by
\[
	 \mathcal M:=\overline{\alg(Z)^{(n)}u},
\] where $u$ is the column vector $(u_1 \cdots u_n)^T$ in  $\cH^{(n)}$.
 By construction, the subspace $\mathcal M$ is invariant for every operator in $\alg(Z)^{(n)}$, but it is not invariant for the operator $W^{(n)}.$ 

Consider the orthogonal decomposition $\mathcal M\oplus \mathcal M^{\perp}$ of $\cH^{(n)}$. With respect to this decomposition, define $\sigma:\cH^{(n)}\rightarrow \cH^{(n)}$ by $\sigma:=\alpha 1_{\mathcal M} + \beta 1_{\mathcal M^{\perp}}$, where $\alpha>\beta$ are positive numbers to be specified. Since $\mathcal M$ is invariant for each $Z_i^{(n)}$ but not $W^{(n)},$ with respect to the above orthogonal decomposition, we may write
\[
Z_i^{(n)}=\bbm
		A_i & B_i  \\
		0 & D_i
	\ebm
\]  and
\[
W^{(n)}=\bbm
		\tilde{A} & \tilde{B}  \\
		\tilde{C} & \tilde{D}
	\ebm,
\] where $\tilde{C}\neq 0$.
Therefore, since a routine computation shows that
\[
\sigma^{-1} Z_i^{(n)}\sigma=\bbm
		A_i & \frac{\beta}{\alpha}B_i  \\
		0 & D_i
	\ebm
\] and 
\[
\sigma^{-1} W^{(n)}\sigma=\bbm
		\tilde{A} & \frac{\beta}{\alpha}\tilde{B}  \\
		\frac{\alpha}{\beta}\tilde{C} & \tilde{D}
	\ebm,
\] we can choose $\alpha>\beta$ so that $\Vert \sigma^{-1}Z^{(n)}\sigma\Vert<1$ and $\Vert \sigma^{-1}W^{(n)}\sigma\Vert>1$. The result follows by letting $s:\cH\rightarrow \cH^{(n)}$ be defined by $s=\sigma U$ for a unitary $U:\cH\rightarrow\cH^{(n)}.$
\end{proof}

By modifying the proof of Theorem 3.2 to apply to an operator tuple $Z$ in a general $B_{\delta}(\B(\cH))$, rather than the particular case of the unit polydisk $\B_1(\cH)^d$, we can prove the following corollary. Corollary 3.3 will imply our main result on weak algebraicity. In the statement of Corollary 3.3, note the reference to a particular term of the exhausting sequence given in (2.2). We need this extra detail in the statement of the result in order to make use of the fact that operator NC functions on $B_{\delta}(\B(\cH))$ are bounded on each term $B_{\delta}(\B(\cH))^k$.

\begin{corollary}
	If $Z \in B_{\delta}(\B(\cH))^k$ and $W\notin \overline{\alg(Z)}^{{}_{WOT}}$ there exists a positive integer $n$ and invertible $s:\cH\rightarrow \cH^{(n)}$ 
	so that $s^{-1}Z^{(n)}s\in B_{\delta}(\B(\cH))^{k+1}$ and $\Vert s^{-1}W^{(n)}s\Vert>1$.
\end{corollary}

\begin{proof}
	By the proof of Theorem 3.2, find positive integer $n$ and a closed subspace $\mathcal M$ of $\cH^{(n)}$ that is invariant for every operator in $\alg(Z)^{(n)}$ but not for $W^{(n)}$. Consider also the map $\sigma:\cH^{(n)}\rightarrow \cH^{(n)}$ in that proof for $\alpha >\beta$ chosen so that the inequalities
 \[
\Vert\delta(\sigma^{-1}Z^{(n)}\sigma)\Vert \leq 1-\frac{1}{k+1}, \,\,\,\,
\Vert \sigma^{-1}Z^{(n)}\sigma\Vert \leq k+1, \,\,\,\,
\Vert \sigma^{-1}W^{(n)}\sigma\Vert>1
\] all hold. This is possible because $Z$ lies in the $k$th term of the exhausting sequence given in (2.2).
Therefore, $s^{-1}Z^{(n)}s\in B_{\delta}(\B(\cH))^{k+1}$ and $\Vert s^{-1}W^{(n)}s\Vert>1$ where $s=\sigma U$ for some $U:\cH\rightarrow\cH^{(n)}$ unitary.  
\end{proof}

With Corollary 3.3 established, we are now in a position to prove our main result.

\begin{proof}[Proof of Theorem 3.1]
	Suppose for contradiction that there is $Z\in B_{\delta}(\B(\cH))$ such that $$W:=\vf(Z)\notin \overline{\alg(Z)}^{{}_{WOT}}.$$ Since the sequence $(B_{\delta}(\B(\cH))^k)_k$ exhausts $B_{\delta}(\B(\cH))$, there is some integer $k$ such that $Z\in B_{\delta}(\B(\cH))^k$. Without loss of generality, we may assume that $\Vert \vf \Vert \leq 1$ on $B_{\delta}(\B(\cH))^{k+1}$. By Corollary 3.3, there is a positive integer $n$ and invertible $s:\cH\rightarrow \cH^{(n)}$ so that $s^{-1}Z^{(n)}s\in B_{\delta}(\B(\cH))^{k+1}$ and $\Vert 
	s^{-1}W^{(n)}s\Vert>1$. We now arrive at a contradiction by using that $\vf$ is operator NC: $$\Vert\vf(s^{-1}Z^{(n)}s)\Vert=\Vert s^{-1}W^{(n)}s\Vert >1.$$  This completes the proof.
\end{proof}


\begin{thebibliography}{99}

\bibitem{abd}
{\sc G. Abduvalieva, D.~S. Kaliuzhnyi-Verbovetskyi}.
\newblock {\it Implicit/inverse function theorems for free noncommutative functions},
\newblock  J. Funct. Anal. {\bf 269}(2015), 2813--2844.


\bibitem{amop}
{\sc J. Agler, J.~E. McCarthy}.
\newblock {\it Non-commutative holomorphic functions on operator domains},
\newblock  Eur. J. Math. {\bf 1}(2015), 731--745.

\bibitem{global1}
{\sc J. Agler, J.~E. McCarthy}.
\newblock {\it Global holomorphic functions in several non-commuting variables},
\newblock Can. J. Math. {\bf 67}(2015), 241--285.

\bibitem{global2}
{\sc J. Agler, J.~E. McCarthy}.
\newblock {\it Global holomorphic functions in several non-commuting variables II},
\newblock Can. Math. Bull. {\bf 61}(2018), 458--463.

\bibitem{opanalysis}
{\sc J. Agler, J.~E. McCarthy, N. Young}.
\newblock {\it Operator Analysis: Hilbert Space Methods in Complex Analysis}, Cambridge Tracts in Mathematics, Cambridge University Press, 2020. 

\bibitem{rowball}
{\sc M. Augat, M.~T. Jury, J.~E. Pascoe}.
\newblock {\it Effective noncommutative Nevanlinna-Pick interpolation in the row ball, and applications},
\newblock https://arxiv.org/pdf/2005.07556.pdf.

\bibitem{opNC}
{\sc M. Augat, J.~E. McCarthy}.
\newblock {\it Operator NC functions},
\newblock https://arxiv.org/pdf/2107.13079.pdf.

\bibitem{conway}
{\sc J.~B. Conway}.
\newblock {\it A Course in Functional Analysis},
\newblock Graduate Texts in Mathematics, Springer-Verlag New York, 2007.

\bibitem{HKM}
{\sc J.~W. Helton, I. Klep, S. McCullough},
\newblock {\it Proper analytic free maps},
\newblock  J. Funct. Anal. {\bf 260}(2011), 1476--1490.

\bibitem{JKMMP}
  {\sc M.~T. Jury, I. Klep, M.~E. Mancuso, S. McCullough, J.~E. Pascoe}. 
\newblock {\it Noncommutative partial convexity via $\Gamma$-convexity}, 
 J. Geom. Anal.
{\bf 31}(2021), 3137--3160.

\bibitem{book}
{\sc D.~S. Kaliuzhnyi-Verbovetskyi, V. Vinnikov}.
\newblock {\it Foundations of free noncommutative function theory}, vol.~199 of
  Mathematical Surveys and Monographs,
\newblock American Mathematical Society, Providence, RI, 2014.

\bibitem{inverse}
{\sc M.~E. Mancuso}. 
\newblock {\it Inverse and implicit function theorems for noncommutative functions on operator domains}, 
\newblock J. Oper. Theory {\bf 83}(2020), 447--473.

\bibitem{james}
{\sc J.~E. Pascoe}.
\newblock {\it The inverse function theorem and the Jacobian conjecture for free analysis},
\newblock Math. Z. {\bf 278}(2014), 987--994.

\bibitem{pop1}
{\sc G. Popescu}.
\newblock {\it Free holomorphic automorphisms of the unit ball of {$B( H)^n$}},
\newblock J. Reine Angew. Math. {\bf 638}(2010), 119--168.

\bibitem{pop2}
{\sc G. Popescu}.
\newblock {\it Free holomorphic functions on the unit ball of {$B(H)^n$}, ii},
\newblock J. Funct. Anal. {\bf 258}(2010), 1513 -- 1578.




\end{thebibliography}
\end{document}